\documentclass{amsart}

 \newtheorem{thm}{Theorem}[section]
 \newtheorem{cor}[thm]{Corollary}
 \newtheorem{lem}[thm]{Lemma}
 \newtheorem{prop}[thm]{Proposition}
 \theoremstyle{definition}
 \newtheorem{defn}[thm]{Definition}
 \newtheorem{rem}[thm]{Remark}

 \numberwithin{equation}{section}
 
 \DeclareMathOperator{\Span}{Span}

\begin{document}

\title[Integral Lattices of the $SU(2)$-TQFT-Modules]
{Integral Lattices of the $SU(2)$-TQFT-Modules}
\author{Khaled Qazaqzeh}

\address{Department of Mathematics, Yarmouk University, Irbed 21163,
Jordan}
\email{qazaqzeh@yu.edu.jo}
\urladdr{http://faculty.yu.edu.jo/qazaqzeh/}


\thanks{I thank the ICTP and Yarmouk University for supporting me to present this work at the international conference on Quantum Topology Institute of Mathematics, VAST, Hanoi, Vietnam, August 6-12, 2007}


\begin{abstract}

We find bases for naturally defined lattices over certain
rings of integers in the $SU(2)$-TQFT-theory modules  of surfaces.
We consider the TQFT where the Kauffman's $A$ variable is a  root of
unity of order four times an odd prime. As an application, we show
that the Frohman Kania-Bartoszynska ideal invariant for 3-manifolds with boundary
using the $SU(2)$-TQFT-theory
is equal to the product of the ideals using the $2^{'}$-theory and the
$SO(3)$-TQFT-theory under a certain change of coefficients.

\end{abstract}

\maketitle

\section*{Introduction}
In this paper, we let $p$ denote an odd prime or twice an odd prime unless specified otherwise. Also, we let $\Sigma$
denote a connected surface of genus $g$ without colored points.

Based on the integrality results of the Witten-Reshetikhin-Tureav
quantum invariants of closed 3-manifolds ~\cite{GR97,M}, an integral
functor $\mathcal{S}_{p}$ is defined in ~\cite{G04}. This is
a functor that associates to a surface $\Sigma$, a lattice
$\mathcal{S}_{p}(\Sigma)$ over the cyclotomic ring of
integers

\[ \mathcal{O}_{p} =
  \begin{cases}
  \mathbb{Z}[A_{p}],
     & \text{if $p \equiv  \quad -1 \pmod {4}$},
     \\
    {\mathbb{Z}[\alpha_{p}]}, & \text{if $p \equiv 1 $ or $ 2 \pmod
    {4}$},
  \end{cases}
\]
here and elsewhere $A_{p}$, $\alpha_{p}$ are $\zeta_{2p}$ and
$\zeta_{4p}$ respectively for $p\geq 3$.

Gilmer in ~\cite{G04} showed that these
lattices are free in the case of $p$ is an odd prime
and projective in the case of $p$ is twice an odd
prime. In ~\cite{GMW04}, the authors gave explicit bases for
these lattices in genus one and two at roots of unity of odd prime order.
Recently, Gilmer and Masbaum announced a basis for $\mathcal{S}_{p}(\Sigma)$
and hence they gave an independent proof of freeness in the case of $p$ is an
odd prime.

In the 2-theory, the author showed that the lattice $\mathcal{S}_{2}(\Sigma)$
is free by constructing an explicit basis. Also in \cite{Q}, he
showed that the lattice $\mathcal{S}_{p}(S^{1}\times S^{1})$ is free by constructing
two explicit bases in the case of $p$ is twice an odd prime. In this paper, we consider
the prime version of the 2-theory and give an explicit basis for the lattice
$\mathcal{S}^{'}_{2}(\Sigma)$ 
that plays a crucial role in constructing a basis for $\mathcal{S}_{p}(\Sigma)$ for $p$ twice an odd prime.

Frohman and Kania-Bartoszynska in ~\cite{FK} defined an ideal
invariant of 3-manifolds with boundary using the $SU(2)$-TQFT-theory
that is not finitely generated by definition.
In fact, they make use of another ideal that they defined to give an
estimate for this ideal. Later on, Gilmer and Masbaum in ~\cite{GM04}
defined an analogous ideal invariant using the $SO(3)$-TQFT-theory for
3-manifolds with boundary. Also, they showed that this ideal
is finitely generated by giving a finite set
of generators.
In this paper, we give a similar result concerning the ideal using the
$SU(2)$-TQFT-theory. Moreover, we show that this ideal is equal to the product
of the ideals using the $2^{'}$- and the $SO(3)$-TQFT-theories.

In the first section, we review the integral TQFT-functor that was
first introduced in ~\cite{G04}. The quantization functor for $p=2$ is
discussed in the next section, following ~\cite{BHMV3}. Also in this
section, we give a basis for $\mathcal{S}^{'}_{2}(\Sigma)$.
We reformulate some of the results given in
\cite{BHMV3} concerning the relation between the $SO(3)$- and the $SU(2)$-TQFTs
in the third section to serve our need.
In the fourth section,
we list some of the results concerning the surface $S^{1}\times S^{1}$ stated in \cite{GM04,GMW04,Q}.
The main result will be given in the fifth section. In the last section, we give an application of our main result concerning the Frohman Kania-Bartoszynska ideal.


\section{The Integral TQFT-functor}
We consider the $(2+1)$-dimensional TQFT constructed as the main
example of ~\cite{BHMV3} with some modifications. In
particular, we use the cobordism category $\mathcal{C}$ discussed
in ~\cite{G04,GQ} where the 3-manifolds have banded links and
surfaces without colored points. Hence the objects are
oriented surfaces with extra structure (Lagrangian subspaces of
their first real homology). The cobordisms are equivalence classes
of compact oriented 3-manifolds with extra structure (an integer
weight) with banded links sitting inside of them. Two cobordisms
with the same weight are said to be equivalent if there is an
orientation preserving diffeomorphism that fixes the boundary.

We restrict this section for the case where $p\geq 3$ and the case where $ p = 2$ will
discussed in the next section.

Now, we consider the TQFT-functor ($V_{p},Z_{p}$) from
$\mathcal{C}$  to the category of finitely generated free
$k_{p}$-modules, where

\[ k_{p} =
  \begin{cases}
  \mathbb{Z}[A_{p},\frac{1}{p}],
     & \text{if $p \equiv  \quad -1 \pmod{4}$};
     \\
    {\mathbb{Z}[\alpha_{p},\frac{1}{p}]}, & \text{if $p \equiv 1 $ or $2
    \pmod{4}$}.
  \end{cases}
\]

The functor ($V_{p},Z_{p}$) is defined as
follows. $V_{p}(\Sigma)$ is a quotient of the $k_{p}$-module
generated by all cobordisms with boundary $\Sigma$, and $Z_{p}(M)$
is the $k_{p}$-linear map from $V_{p}(\Sigma)$ to
$V_{p}(\Sigma^{'})$ (where $\partial M = -\Sigma\coprod
\Sigma^{'}$) induced by gluing representatives of elements of
$V_{p}(\Sigma)$ to $M$ along $\Sigma$ via the identification map.

If $M$ is a closed cobordism, then $Z_{p}[M]$ is the
multiplication by the scalar $ \langle M \rangle_{p}$ defined in
~\cite[Section.\,2]{BHMV3}. This invariant is normalized in two other
ways. The first normalization of this invariant is
$I_{p}(M)=\mathcal{D}_{p}\langle M_{\flat}\rangle_{p}$. Here and
elsewhere $M_{\flat}$ is the 3-manifold $M$ with a reassigned
weight zero, and $\mathcal{D}_{p}=\langle
S^{3}_{\flat}\rangle_{p}^{-1}$. The second normalization is
$\theta_{p}(M)=\mathcal{D}_{p}^{\beta_{1}(M)+1}\langle
M_{\flat}\rangle_{p}$, i.e
$\theta_{p}(M)=\mathcal{D}_{p}^{\beta_{1}(M)}I_{p}(M)$.

If $\partial M = \Sigma$ and $M$ is considered as a cobordism from
$\emptyset$ to $\Sigma$, then $Z_{p}(M)(1)\in V_{p}(\Sigma)$ is
denoted by $[M]_{p}$ and called a \textsl{vacuum state} and it is
\textsl{connected} if $M$ is connected. Finally, note that $V_{p}$
is generated over $k_{p}$ by all vacuum states.

The modules $V_{p}(\Sigma)$ are free modules over $k_{p}$, and carry a nonsingular
Hermitian sesquilinear form
 \[
 \langle \ , \ \rangle_{\Sigma} : V_{p}(\Sigma)\times V_{p}(\Sigma)
 \rightarrow k_{p},
 \]
given by
 \begin{equation}\label{e:form1}
 \langle[M_{1}],[M_{2}]\rangle_{\Sigma} = \langle
 M_{1}\cup_{\Sigma}-M_{2}\rangle_{p}.
 \end{equation}
Here -$M$ is the cobordism $M$ with the orientation reversed and
multiplying the integer weight by -1, and leaving the Lagrangian
subspace on the boundary the same.

A standard basis $\{u_{\sigma}\}$ for $V_{p}(\Sigma)$ is given
(see ~\cite{BHMV3}) in terms of $p$-admissible colorings $\sigma$
of a banded uni-trivalent graph in a handlebody of genus $g$ whose boundary is $\Sigma$.


Let $\mathcal{C^{'}}$ be the subcategory of $\mathcal{C}$
consisting of the nonempty connected surfaces and connected
cobordisms between them.

\begin{defn}
For the surface $\Sigma$, we define $\mathcal{S}_{p}(\Sigma)$ to be the
$\mathcal{O}_{p}$-submodule of $V_{p}(\Sigma)$ generated by all
connected vacuum states.
\end{defn}
If $M :\Sigma \rightarrow \Sigma^{'}$ is a cobordism of
$\mathcal{C}^{'}$, then $Z_{p}(M)([N]_{p}) = [M
\cup_{\Sigma}-N]_{p} \in \mathcal{S}_{p}(\Sigma{'})$. Hence we
obtain a functor from $\mathcal{C}^{'}$ to the category of
$\mathcal{O}_{p}$-modules. These modules are projective as they
are finitely generated and torsion-free over Dedekind domains
~\cite[Theorem.\,2.5]{G04}. Also, these modules carry
 an $\mathcal{O}_{p}$-Hermitian sesquilinear form
\[(\ ,\ )_{\Sigma}:\mathcal{S}_{p}(\Sigma)\times \mathcal{S}_{p}(\Sigma) \rightarrow
\mathcal{O}_{p},\] given by
\begin{equation}\label{e:form2}
 ([M_{1}],[M_{2}])_{\Sigma} =
\mathcal{D}_{p}\langle[M_{1}],[M_{2}]\rangle =
\mathcal{D}_{p}\langle M_{1}\cup_{\Sigma}-M_{2}\rangle_{p},
\end{equation}

The value of this form always lies in $\mathcal{O}_{p}$ by the
integrality results for closed 3-manifolds in ~\cite{GR97,M}.

If $R\subseteq \mathcal{S}_{p}(\Sigma)$ is an $\mathcal{O}_{p}-$submodule
define
\[R^{\sharp}=\{v\in V_{p}(\Sigma)| (r,v)_{\Sigma}\in\mathcal{O}_{p}, \forall r\in R\},\]
then we can conclude
\begin{equation}\label{e:dual}
R\subseteq \mathcal{S}_{p}(\Sigma)\subseteq \mathcal{S}_{p}^{\sharp}
(\Sigma)\subseteq R^{\sharp}.
\end{equation}
\begin{defn}
A Hermitian sesquilinear form on a projective module over a
Dedekind domain is called  $\textit{non-degenerate}$ if the
adjoint map is injective, and $\textit{unimodular}$ if the adjoint
map is an isomorphism.
\end{defn}
For our use, if the matrix of the form has a nonzero (unit)
determinant, then the form will be non-degenerate (unimodular)
respectively.

The elements of the standard basis $u_{\sigma}$ defined in \cite{BHMV3} lie in
$\mathcal{S}_{p}(\Sigma)$ when $p$ is twice an odd prime. This
follows from the fact that the quantum integers (denominators of
the Jones-Wenzl idempotents) are units in $\mathcal{O}_{p}$
(see ~\cite[Corollary.\,6.4]{Q}). An admissible colored uni-trivalent graph
~\cite{BHMV3} is to be interpreted, here and elsewhere, as an
$\mathcal{O}_{p}$-linear combination of links.

A $p$-admissibly uni-trivalent colored graph whose simple closed curves may be cabled with $\omega_{p}$ in a fixed connected 3-manifold $M$ whose
boundary is $\Sigma$ called a \textsl{mixed graph}. Here and elsewhere
\[
[i]_{p} = (A_{p}^{2i} - A_{p}^{-2i})/ (A_{p}^{2} - A_{p}^{-2}),
\]
and
\[
\omega_{p} =  \mathcal{D}_{p}^{-1}\sum_{i=0}^{d_{p} -1}(-1)^{i}[i+1]_{p}  e_{i},
\]
where $d_{p} = [(p-1)/2]$.

Since the surgery axiom (S2) in ~\cite{BHMV3} holds,
we can choose this fixed 3-manifold to be a handlebody whose boundary is $\Sigma$.
Now we can describe the elements of the module $\mathcal{S}_{p}(\Sigma)$ in terms
of the mixed graph notation as follows:
\begin{prop}\label{p:mixed}
A mixed graph
represents an element in $\mathcal{S}_{p}(\Sigma)$. Moreover,
$\mathcal{S}_{p}(\Sigma)$ is generated over $\mathcal{O}_{p}$ by
all the elements given by a mixed graph in a fixed handlebody
whose boundary is $\Sigma$ with the same genus.
\end{prop}
\begin{proof}
The first statement follows from that fact that $V_{p}$ satisfies
the second surgery axiom. The second statement follows from the fact that
every 3-manifold with boundary $\Sigma$ is obtained by a sequence
of 2-surgeries to a handlebody of the same boundary and the
definition of $\mathcal{S}_{p}(\Sigma)$.
\end{proof}


\section{The quantization functor for $p=2$}
To relate between the $SU(2)$-TQFT and the $SO(3)$-TQFT theories,
we need to discuss the $2^{'}$-theory.

We start by recalling the ring used in this theory and its
ring of integers:
\[
k_{2}=\mathbb{Z}[\alpha_{2},\frac{1}{2}] \quad \text{and}\quad \mathcal{O}_{2} =
\mathbb{Z}[\alpha_{2}].
\]

The surgery element for this 2-theory is $\omega_{2}=
\frac{1}{\sqrt{2}}\Omega_{2}$ where $\Omega_{2} = 1 + \frac{z}{2}$ as defined in \cite[Section.\,6]{BHMV1}.
One has $\mathcal{D}_{2}=\sqrt{2}$, and
$\kappa_{2}=\zeta_{8}$. Therefore, the invariant of a closed
connected 3-manifold $M$ that is obtained by doing surgery on
$S^{3}$ along the link $L$ in the complement of a banded link $K$ is given by
\begin{equation}\label{e:invariant2}
 \langle M, K \rangle^{'}_{2}
= (-2)^{-n}\frac{1}{\sqrt{2}}<L(\omega_{2})\cup K>, \
\end{equation}
where $< L(\omega_{2})\cup K>$ denotes the Kauffman bracket of the disjoint union of the link $L$ cabled with $\omega_{2}$
and $K$, and $ n $ is the number of components of the banded link $K$.
From this formula, we can easily verify that
\begin{equation}\label{e:sum}
 \langle M_{1}\sharp
M_{2}\rangle^{'}_{2} = \sqrt{2}\langle M_{1}\rangle^{'}_{2}\langle
M_{2}\rangle^{'}_{2}.
 \end{equation}
Now this invariant $\langle M\rangle^{'}_{2}$ defined in ~\cite[Section.\,2]{BHMV3} is involutive and extended to be multiplicative, hence
(by ~\cite[Proposition.\,1.1]{BHMV3}) there exits a unique cobordism
generated quantization functor that extends $\langle M\rangle^{'}_{2}$
which is denoted by ($V^{'}_{2},Z^{'}_{2}$). The modules $V^{'}_{2}(\Sigma)$
carry a Hermitian sesquilinear form defined as follows.
\[
\langle \ , \ \rangle_{\Sigma}:V^{'}_{2}(\Sigma)\times
V^{'}_{2}(\Sigma)\rightarrow k_{2},
\]given by
\[
 \langle[M_{1}],[M_{2}]\rangle_{\Sigma} = \langle
M_{1}\cup_{\Sigma}-M_{2}\rangle^{'}_{2} .\]

\begin{rem}
The elements in $V_{2}^{'}(\Sigma)$ can not be represented by elements of the Kauffman
bracket skein module of a handlebody at $A = i$. This because the skein relations do not hold
for elements represented by banded links $K$ in a 3-manifold because of the undetermined way of the number
of components of $K$ change under smoothing. But elements can be represented by formal $k_{2}$-linear combinations of links in a fixed handlebody $H$ whose boundary is $\Sigma$. 
\end{rem}

We let $\omega_{2}^{'}$ denote the surgery element for the $2^{'}$-theory. We have $\omega_{2}^{'} = \frac{1}{\sqrt{2}}(1-z)$.
By \cite[1.5 and 6.3]{BHMV3}, $V^{'}_{2}(S^{1}\times S^{1})$ is generated
by two elements, named $1$ and $z$, each of which is a solid torus where
the core is colored either 0 or 1 respectively.
If we restrict this theory to the category of nonempty connected
objects and connected cobordisms between them, then we have an
integral cobordism theory as before. This follows from the fact
$\sqrt{2}\left\langle \  \right\rangle^{'}_{2}$ is integral as stated in \cite[Lemma.\,6.2]{GR97}.

\begin{defn}
We define $\mathcal{S}^{'}_{2}(\Sigma)$ to be the
$\mathcal{O}_{2}$-submodule of $V^{'}_{2}(\Sigma)$ generated by all
connected vacuum states, and we define an
$\mathcal{O}_{2}$-Hermitian sesquilinear form on
$\mathcal{S}^{'}_{2}(\Sigma)$ given by $(\ , \
)_{\Sigma}=\sqrt{2}\langle \ ,\ \rangle_{\Sigma}$.
\end{defn}

The above basis for $V^{'}_{2}(S^{1}\times S^{1})$ does not generate
$\mathcal{S}^{'}_{2}(S^{1}\times S^{1})$. The following theorem gives
a basis for $\mathcal{S}^{'}_{2}(S^{1}\times S^{1})$. 

\begin{thm}\label{t:genus1}
The set $\mathcal{B} = \{\omega_{2}^{'}, t(\omega_{2}^{'})\}$
is a basis for $\mathcal{S}^{'}_{2}(S^{1}\times S^{1})$, and the form $(\ , \
)_{S^{1}\times S^{1}}$ is
unimodular on $\mathcal{S}^{'}_{2}(S^{1}\times S^{1})$. 
\end{thm}
\begin{proof}
From the definition we know that these two elements lie in $\mathcal{S}^{'}_{2}(S^{1}\times S^{1})$,
hence $\mathcal{W}=\Span_{\mathcal{O}_{2}}\mathcal{B}\subseteq
\mathcal{S}^{'}_{2}(S^{1}\times S^{1})$. The matrix of the form $(\ ,\
)_{S^{1}\times S^{1}}$ in terms of $\mathcal{B}$ is given by
\[
\left(%
\begin{array}{cc}
  \sqrt{2} & \frac{1-i}{\sqrt{2}} \\
  \frac{1+i}{\sqrt{2}} & \sqrt{2} \\
\end{array}%
\right) .\] 

This follows from \cite[Theorem.\,5.2]{Q}, so the form restricted on $\mathcal{W}$ has a unit
determinant. Hence $\mathcal{W} = \mathcal{W}^{\sharp}$. Using the analog in the $2^{'}$-theory of 
equation (\ref{e:dual}), we get that $\mathcal{W}$ is all of
$\mathcal{S}^{'}_{2}(S^{1}\times S^{1})$. In conclusion,
$\{\omega_{2}^{'},t(\omega_{2}^{'})\}$ is a basis for
$\mathcal{S}^{'}_{2}(S^{1}\times S^{1})$.
\end{proof}

\begin{defn}
Let $H_{i_{1}i_{2}\ldots i_{g}}$ stands for the element represented by the handlebody $H$ of genus $g$
after doing cabling along the core of the $j$-th hole with $t^{i_{j}}(\omega_{2}^{'})$.
\end{defn}

\begin{thm}\label{2theory}
The set $\{H_{i_{1}i_{2}\ldots i_{g}} | i_{j} = 0, 1, 1\leq j \leq g \}$ forms a basis for $\mathcal{S}^{'}_{2}(\Sigma_{g})$
\end{thm}
\begin{proof}
Let $(S^{1}\times S^{2})_{ij}$ denote $S^{1}\times S^{2}$ formed
by gluing two solid tori after cabling their cores with
$t^{i}(\omega_{2}^{'})$, and $t^{j}(\omega_{2}^{'})$ where $i,j\in
\{0,1\}$. Also, let $H_{1}, H_{2}$ be any two elements in the above set.
Let us look at the pairing

\begin{align*}
(H_{1}, H_{2})_{\Sigma} & = \sqrt{2}\langle H_{1},H_{2}\rangle_{\Sigma}\\ & =
\sqrt{2}\langle \sharp_{k=1}^{g}(S^{1}\times S^{2})_{i_{k}j_{k}}\rangle_{2}^{'} \\
 &= \sqrt{2}^{g} \prod_{k=1}^{g} \langle
(S^{1}\times S^{2})_{i_{k}j_{k}}\rangle_{2}^{'}\\
&=\prod_{k=1}^{g}(t^{i_{k}}(\omega_{2}^{'}),
t^{j_{k}}(\omega_{2}^{'}))_{S^{1}\times S^{2}}.
\end{align*}
With a natural order, the matrix of the form in terms of this set
is given by $\bigotimes^{g}B$ ($B$ is defined in the proof of the
previous theorem). This implies that the determinant of this form
is a unit in $\mathcal{O}_{2}$. By a similar argument as in the proof of Theorem
\ref{t:genus1}, the module generated by this set is all of
$\mathcal{S}^{'}_{2}(\Sigma_{g})$.
\end{proof}


\section{Relating the $SU(2)$-TQFT and the $SO(3)$-TQFT theories}

For the remainder of the paper, $r$ will denote an odd prime and $p = 2r$. The results of this section are slight variations of the results of
~\cite[Section.\,6]{BHMV3} and ~\cite[Section.\,2]{BHMV2}. The ring $k_{p}$ is
not exactly the same as the ring denoted this way in
~\cite{BHMV3}.

We consider the ring $k_{p}$ as a $k_{2}$ (or a
$k_{r}$)-module via the homomorphisms defined below that are
slight variation of the maps defined in ~\cite[Section.\,2]{BHMV2}.
The reason for this variation is that we use the cobordism category of extended
surfaces and 3-manifolds instead of $p_{1}$-structures
We list the results that we use later
(see ~\cite[Section.\,6]{Q} for proofs and more details).
\begin{lem}\label{l:maps}
There are well-defined ring homomorphisms $i_{r}:k_{2}\rightarrow
k_{p}$, $j_{r}: k_{r}\rightarrow k_{p}$  given by

 \[
i_{r}(\alpha_{2}) = \alpha_{p}^{r^{2}}, \ j_{r}(\alpha_{r}) =
 \alpha_{p}^{1+r^{2}} \text{for}\ r\equiv 1\pmod{4}, \text{and}\] \[
 j_{r}(A_{r})=A_{p}^{1+r^{2}} \text{for} \ r\equiv -1\pmod{4}.
 \]

\end{lem}

The above ring homomorphisms allow one to construct a $k_{p}$-module starting with
any $k_{2}$-module or $k_{r}$-module by tensoring the original module
with $k_{p}$ over $k_{2}$ or $k_{r}$ respectively.
We let ${\widehat{V}}^{'}_{2}(\Sigma)$
or ${\widehat{V}}_{r}(\Sigma)$ be the $k_{p}-$module obtained in
this way. We give a relation between
${\widehat{V}}^{'}_{2}(\Sigma),{\widehat{V}}_{r}(\Sigma),$ and ${V}_{2r}(\Sigma)$ for any
surface $\Sigma$.
\begin{thm}\label{t:tensor}
There is a natural $k_{p}$-isomorphism $F:
{V}_{p}(\Sigma)\rightarrow {\widehat{V}}^{'}_{2}(\Sigma)\otimes_{k_{p}}
{\widehat{V}}_{r}(\Sigma)$ defined by
\begin{equation}\label{e:map}
F([M]_{p}) = [M]^{'}_{2}\otimes_{k_{p}}[M]_{r},
\end{equation}
where $M$ is a 3-manifold with banded link (but not linear
combination of links) sitting inside of it.
\end{thm}

The proof of the above theorem is given in \cite{BHMV3}. Further details can be found in \cite[Section.\,6]{Q}.


\section{Basis for $\mathcal{S}_{p}(S^1\times S^{1})$}

In ~\cite{GMW04}, the authors gave two explicit bases for
$\mathcal{S}_{p}(S^1\times S^{1})$ where $p$ is an odd prime.
To serve our need, we state the first basis.

\begin{thm}
The set $ \{ 1, v, \ldots, v^{d_{r}-1}\}$ forms a basis for
$\mathcal{S}_{r}(S^{1}\times S^{1})$, where $v=\frac{z+2}{1+\zeta_{p}}$.
\end{thm}
We recall some results from ~\cite{Q} that will be used later.

\noindent
\begin{bf}{Notation.}
\end{bf} We use the notation from \cite{Q}: $\mathcal{\widehat{S}}_{i}(\Sigma) = \mathcal{S}_{i}(\Sigma)\otimes_{k_{i}}k_{p}$ for $i = 2$ or $r$.

\begin{thm}
The set $  \{ t^{i+ \delta_{i} p}(\omega_{p}) \ | \
0\leq i \leq d_{p}-1 \}$ forms a basis for
$\mathcal{S}_{p}(S^{1}\times S^{1})$,
where
\[
 \delta_{i} =
  \begin{cases}
  0,
     & \text{if $i + p \equiv \quad 2 \quad or \quad 3  \pmod
     {4}$};
     \\
    1, & \text{if $i + p \equiv \quad 0 \quad or \quad 1  \pmod
    {4}$}.
  \end{cases}
\]
\end{thm}

As a corollary of the proof of the above theorem, we obtain:
\begin{cor}\label{isomophism}
The  isomorphism $F$ in Theorem \ref{t:tensor} induces an isomorphism between
$\mathcal{S}_{p}(S^{1}\times
S^{1})$ and $ {\mathcal{\widehat{S}}}^{'}_{2}(S^{1}\times S^{1}) \otimes_{k_{p}}
{\mathcal{\widehat{S}}}_{r}(S^{1}\times S^{1})$.
\end{cor}

Therefore, we can define elements $u_{ij}$ so that $F(u_{ij}) = t^{i}(\omega_{2}^{'})\otimes v^{j}
$ for $ 0 \leq j \leq d_{r}-1$ and $i = 0, 1$.
Using the proof of \cite[Theorem.\,7.6]{Q}, one sees that
\[
u_{0j} = \sum_{\stackrel{i = 0}{{i+\delta_{i}p \equiv 0 \pmod{4}}}}^{d_{p} - 1}k_{i}t^{{i+\delta_{i}p}}(\omega_{p}),
\]
and
\[
u_{1j} = \sum_{\stackrel{i = 0}{{i+\delta_{i}p \equiv 1 \pmod{4}}}}^{d_{p} - 1}k^{'}_{i}t^{{i+\delta_{i}p}}(\omega_{p}),
\]
for $k_{i}, k^{'}_{i} \in \mathcal{O}_{p}$.

Hence the set $\{u_{ij} \ | \ 0 \leq j \leq d_{r}-1, i = 0, 1\}$ forms a basis for $\mathcal{S}_{p}(S^{1}\times S^{1})$.
Also, there is a special element $u$ in $\mathcal{S}_{p}(S^{1}\times S^{1})$
that maps to $1\otimes v \in {\mathcal{\widehat{S}}}^{'}_{2}(S^{1}\times S^{1}) \otimes_{k_{p}}
{\mathcal{\widehat{S}}}_{r}(S^{1}\times S^{1})$. The following lemma gives the explicit formula for $u$.

\begin{lem}
$u = \frac{e_{r-3} + 2}{1+\zeta_{p}}$.
\end{lem}

\begin{proof}
From the above isomorphism, there is only one element that maps to $1\otimes v$. Now it is enough to prove that
$F(\frac{e_{r-3} + 2}{1+\zeta_{p}}) = 1\otimes v.$
\begin{align*}
F(\frac{e_{r-3} + 2}{1+\zeta_{p}})& = \frac{1}{1 + \zeta_{p}}(F(e_{r-3}) + F(2))\\
& = \frac{1}{1 + \zeta_{p}}(1\otimes e_{r-3} + 1\otimes 2)\\
& = \frac{1}{1 + \zeta_{p}}(1\otimes z + 1\otimes 2)\\
& = 1\otimes \frac{z + 2}{1 + \zeta_{p}}\\
& = 1\otimes v.
\end{align*}
where the second equality 
follows from the fact $z^{2k} = 1$ for any integer $k$ in the 2$^{'}$-theory and
the third equality follows from the fact $e_{r-3} = e_{1} = z$ in the $SO(3)$-TQFT-theory.

\end{proof}


\section{Basis for $\mathcal{S}_{p}(\Sigma)$}

We are ready now to state our main theorem that gives a basis for $\mathcal{S}_{p}(\Sigma)$.
The basis is given in terms of a $r$-admissible colorings of a fixed lollipop tree
that was first introduced in ~\cite[Section.\,3]{GM04}.

Let $H_{g}$ be a fixed handlebody whose boundary is $\Sigma$ without colored points.
Also, let $G$ be any uni-trivalent banded graph having the same homotopy
type as the handlebody $H_{g}$.
We recall the definition of a \textsl{lollipop tree}.
\begin{defn}\cite{GM04}
Let $G$ be as above.
Then $G$ is a \textsl{lollipop tree} if $G$ has exactly $g$ loop edges, and the complement of the loop edges in $G$ is a tree and denoted by $T$.

\end{defn}
We describe a special coloring of the above lollipop tree $G$. 
We call an edge that is incident
to a loop edge a \textsl{stick edge}. An edge which is neither a loop nor a stick edge is called \textsl{ordinary}.
The colors of the tree $T$ are denoted by $2a_{1}, 2a_{2}, \ldots, 2a_{g}$ for the
stick edges, and $2c_{1}, 2c_{2}, \ldots $ for the ordinary edges. Here
\[
 0 \leq a_{i}, c_{j} \leq d_{r}-1,
\]

Now we color the loop edges by $a_{i} + b_{i}$ for the loop edge incident to the stick
edge colored $2a_{i}$, where
\[
0 \leq b_{i} \leq d_{r} - 1 - a_{i}.
\]
The case $g = 2, s = 0$ is a special case as we count the stick edge twice
with $a_{1} = a_{2}$.

 Coloring $G$ by $(a, b, c)$,  where $ a = (a_{1}, a_{2}, \ldots, a_{g})$,
$b = (b_{1}, b_{2}, \ldots, b_{g}),$ and $c  = (c_{1}, c_{2}, \ldots)$ such that the above inequalities hold
 and such that the pair $(2a, 2c)$ is an
$r$-admissible coloring of $T$, gives
an element $\textsl{g}(a, b, c)$ in $V_{r}(\Sigma)$. Each such vector
represents a linear combination over $\mathcal{O}_{r}$ of skein elements in
the handlebody $H_{g}$ by replacing the edges of $G$ with the appropriate
Jones-Wenzl idempotents.
\begin{rem}
This procedure is first described by the authors of \cite[Section.\,6]{GM04}. Also, they called
the special element  $\textsl{\textsl{g}}((1, 1), (0, 0))\in \mathcal{S}_{r}(\Sigma_{2})$ an eyeglass.
\end{rem}

Now we want to define $\hat{\textsl{g}}(a, 0, c)$ as an element in $\mathcal{S}_{p}(\Sigma)$ by the following procedure
which may involve choices:
\begin{enumerate}
 \item We start with $\textsl{g}(a, 0, c)$ constructed by the above procedure over $\mathcal{O}_{r}$.
  \item Write $\textsl{g}(a, 0, c)$ as a linear combination over $\mathcal{O}_{r}$ of terms which are colored graphs which
  contain a collection of $\lfloor \frac{1}{2}\sum_{i}a_{i}\rfloor$ eyeglasses  (using the process described in the proof of \cite[Proposition.\,6.2]{GM04}).
  \item Rewrite these eyeglasses as linear combination of links obtained by rewriting each loop as $e_{r-3}$.
  \item Extend the coefficients to $\mathcal{O}_{p}$ by applying $j_{r}$.
\end{enumerate}
\textit{The basis for} $\mathcal{S}_{p}(\Sigma)$ will be denoted by $\mathcal{B}$.
It consists of elements indexed by $(a, b, c, m)$ denoted by $\textbf{\textsl{b}}(a, b, c, m)$,
where $m$ is a $g$-tuple of zeros and ones.
They are defined as follows:

\begin{equation}
\textbf{\textsl{b}}(a, b, c, m) = (1 + \zeta_{p})^{ - \left\lfloor \frac{1}{2}\sum_{i}a_{i}\right\rfloor}u_{m_{1}b_{1}}^{\{1\}}\ldots u_{m_{g}b_{g}}^{\{g\}}\hat{\textsl{g}}(a, 0, c),
\end{equation}
where from now on $w^{\{i\}}$ denotes $w$ that encloses the $i$-th hole.

The above multiplication is defined
since the Kauffman bracket skein module is a module over the absolute skein module
by stacking elements on top of each other.
\begin{thm}
If $\Sigma$ is a surface with no colored points, and $p$ twice an odd prime, then  $\mathcal{S}_{p}(\Sigma)$ is free.
\end{thm}

The rest of this section will be devoted to prove this result. The proof goes in two steps:
first by showing that the elements of $\mathcal{B}$ lie in $\mathcal{S}_{p}(\Sigma)$ and then by
showing that they generate it. Hence we conclude that
$\mathcal{S}_{p}(\Sigma)$ is a free lattice.

\begin{lem}
The eyeglass $\hat{\textsl{g}}((1, 1), (0, 0))\in \mathcal{S}_{r}(\Sigma_{2})$ is divisible by $1 + \zeta_{p}$ in $\mathcal{S}_{p}(\Sigma)$.
\end{lem}
\begin{proof}
We let $u^{\{12\}}$ and $e_{r-3}^{\{12\}}$ denote $u$ and $e_{r-3}$ that enclose the first and
second holes respectively, and $ h = 1 + \zeta_{p}$. Now we expand $\hat{\textsl{g}}((1, 1), (0, 0))$
linearly to obtain:
\begin{align*}
\textsl{\textbf{$\hat{g}$}}((1, 1), (0, 0)) &= e_{r-3}^{\{12\}} - \frac{1}{[2]}e_{r-3}^{\{1\}}e_{r-3}^{\{2\}},\\
& = (hu^{\{12\}} - 2) - \frac{1}{[2]}(hu^{\{1\}} - 2)(hu^{\{2\}} - 2), \\
& = hu^{\{12\}} - [2]^{-1}(h^{2}u^{\{1\}}u^{\{2\}} - 2hu^{\{1\}} - 2hu^{\{2\}} + 4 + 2[2]),\\
& = hu^{\{12\}} - [2]^{-1}(h^{2}u^{\{1\}}u^{\{2\}} - 2hu^{\{1\}} - 2hu^{\{2\}} + 2\zeta_{p}^{-1}h^{2}).
\end{align*}
The result follows as $[2]$ is a unit in $\mathcal{O}_{p}$ and the $u$-graphs
lie in $\mathcal{S}_{p}(\Sigma)$. 
\end{proof}

\begin{prop}
The vectors \textbf{\textsl{b}}$(a, b, c, m)$ lie in $\mathcal{S}_{p}(\Sigma)$.
\end{prop}

\begin{proof}
It is enough to consider the case where $b = 0$ since \textbf{\textsl{b}}$(a, b, c, m)$
is obtained from \textbf{\textsl{b}}$(a, 0, c, m)$ by doing surgery along simple loops.
So we need to show that $\hat{\textsl{g}}(a, 0, c)$ is divisible by 
$(1 + \zeta_{p})^{ - \left\lfloor \frac{1}{2}\sum_{i}a_{i}\right\rfloor}$.

We apply the same procedure given in the ~\cite[Proposition.\,6.2]{GM04} to the colored
graph representing $\textsl{g}(a, 0, c)$. Thus  $\hat{\textsl{g}}(a, 0, c)$
is represented as a linear combination over $\mathcal{O}_{p}$ of diagrams with
$  \left\lfloor\frac{1}{2}\sum_{i}a_{i})\right\rfloor$ eyeglasses
as above. Using the previous lemma, we obtain the required result.
\end{proof}


\begin{prop}
For the map $F$ defined in Theorem \ref{t:tensor}, we have
\[
F(\textbf{\textsl{b}}(a, b, c, m)) = H_{m_{1}\ldots m_{g}}\otimes_{k_{p}} \textbf{\textsl{\emph{b}}}(a, b, c)
\]
where $H_{m_{1}\ldots m_{g}}$ is a basis element for $\mathcal{\widehat{S}}_{2}^{'}(\Sigma)$
described in Theorem \ref{2theory} and \textbf{\textsl{\emph{b}}}$(a, b, c)$
is a basis element for $\mathcal{\widehat{S}}_{r}(\Sigma)$ defined in ~\cite[Section.\,4]{GM04}.
\end{prop}
To obtain a proof for the above proposition, we apply the map $F$
to each term of the linear combination that represents $\textbf{\textsl{b}}(a, b, c, m)$.
Hence we conclude that $F(\mathcal{B})$ generates
$\mathcal{\widehat{S}}_{2}^{'}(\Sigma)\otimes_{k_{p}} \mathcal{\widehat{S}}_{r}(\Sigma)$, but it
is clear that $F(\mathcal{S}_{p}(\Sigma)) \subseteq \mathcal{\widehat{S}}_{2}^{'}(\Sigma)\otimes_{k_{p}} \mathcal{\widehat{S}}_{r}(\Sigma)$.
Therefore, we conclude that $\mathcal{B}$ is a basis for $\mathcal{S}_{p}(\Sigma)$

\begin{cor}\label{main}
The isomorphism $F$ in Theorem \ref{t:tensor} induces an isomorphism between
$
\mathcal{S}_{p}(\Sigma)$ and $ {\mathcal{\widehat{S}}}_{2}^{'}(\Sigma)\otimes_{k_{p}} {\mathcal{\widehat{S}}}_{r}(\Sigma)
$
\end{cor}


\section{The Frohman Kania-Bartoszynska ideal}
This ideal was first introduced by Frohman and Kania-Bartoszynska in ~\cite{FK}.

\begin{defn}
Let $N$ be a 3-manifold with boundary, we define
$\mathcal{J}_{p}(N)$ to be the ideal generated over
$\mathcal{O}_{p}$ by
\[
\{ I_{p}(M)| \ \text{where}\ M \ \text{is a closed connected
3-manifold containing}\ N \}.
\]
\end{defn}
The importance of this ideal is in being an invariant of
3-manifolds with boundary and an obstruction to embedding as
stated below (see ~\cite{FK} for more details).
\begin{prop}
The ideal $\mathcal{J}_{p}$ is an invariant of oriented
3-manifolds with boundary.
\end{prop}
\begin{prop}
If $N_{1}, N_{2}$ are an oriented compact 3-manifolds, and $N_{1}$
embeds in $N_{2}$, then
$\mathcal{J}_{p}(N_{2})\subseteq\mathcal{J}_{p}(N_{1})$.
\end{prop}
\begin{rem}
Frohman and Kania-Bartoszynska defined this ideal using the
$SU(2)$-TQFT-theory. Afterward, Gilmer and Masbaum \cite{GM04} defined $\mathcal{J}_{r}(N)$ in the same way using the $SO(3)$-TQFT-theory. Also Gilmer \cite{G05} defined $\mathcal{J}_{2}(N)$ using the $2^{'}$-theory.
\end{rem}

Following his work with Masbaum in the case $p$ an odd prime,
Gilmer observed that $\mathcal{J}_{p}(N)$ is finitely generated
based on his result that $\mathcal{S}_{p}(\Sigma)$ is finitely
generated. We give a
finite set of generators for this ideal as an application of our
main theorem.

\begin{thm}
Let $N$ be an oriented 3-manifold with boundary $\Sigma$, then $\mathcal{J}_{p}(N)$ is finitely
generated by all scalars $([N],\textbf{\textsl{b}})_{\Sigma}$ as $\textbf{\textsl{b}}$ varies over a
basis for $\mathcal{S}_{p}(\Sigma)$.
\end{thm}

\begin{rem}
This result is analogous to ~\cite[Theorem.\,16.5]{GM04}.
\end{rem}

Finally, the following theorem gives the relation between
the Frohman Kania-Bartoszynska ideals using the $SU(2)$-TQFT and the
$SO(3)$-TQFT theories.
\begin{thm}\label{t:ideal}
Let $N$ be an oriented compact 3-manifold with boundary. Then we
have
\[
\mathcal{J}_{p}(N) =
i_{r}(\mathcal{J}_{2}(N))j_{r}(\mathcal{J}_{r}(N)),
\]
where $i_{r}$ and $j_{r}$ are defined as before.
\end{thm}

\begin{proof}
The result follows as a consequence of corollary \ref{main}.
\end{proof}

\bibliographystyle{amsalpha}

\begin{thebibliography}{W}


\bibitem[BHMV1]{BHMV1}
C. Blanchet, N. Habbeger, G. Masbaum, and P. Vogel,
\emph{Three-Manifold Invariants Derived from the Kauffman
Bracket}, \textsl{Topology} \textbf{31} (1992), 685-699.
\bibitem[BHMV2]{BHMV2}
C. Blanchet, N. Habbeger, G. Masbaum, and P. Vogel, \emph{Remarks
on the Three-Manifold Invariants $\theta_{p}$}, \textsl{'Operator
Algebras, Mathematical Physics, and Low Dimensional Topology'}
(NATO Workshop July 1991) Edited by R. Herman and B. Tanbay,
Research Notes in Mathematics \textbf{5} (1993), 39-59.
\bibitem[BHMV3]{BHMV3} C. Blanchet, N. Habbeger, G. Masbaum, and P.
Vogel, \emph{Topological Quantum Field Theories derived from the
Kauffman bracket}, \textsl{Topology} \textbf{34} (1995), 883-927.
\bibitem[FK]{FK}
C. Frohman, J. Kania-Bartoszynska, \emph{A quantum obstruction to
embedding}, \textsl{Math. Proc. Camb. Phil. Soc.}
\textbf{131} (2001), 279-293.
\bibitem [G1]{G04}
P. M. Gilmer, \emph{Integrality For TQFTS}, \textsl{Duke
Mathematical Journal} \textbf{125}(2004), 389-413.
\bibitem[G2]{G05}
P.~M. Gilmer, \emph{On The Frohman Kania-Bartoszynska Ideal},
\textsl{Math. Proc. Camb. Phil. Soc.}, 141, (2006), 265-271.
\bibitem[GM]{GM04}
P.~M. Gilmer, and G. Masbaum, \emph{Integral Lattices in TQFT},
Annales Scientifiques de l'Ecole Normale Superieure, 40, (2007), 815-844.
\bibitem[GMW]{GMW04}
P. M. Gilmer, G. Masbaum, and P. van Wamelen,  \emph{Integral
bases for TQFT modules and unimodular representations of mapping
class groups}, \textsl{Comment. Math. Helv.} \textbf{79} (2004),
260-284.
\bibitem[GQ]{GQ}
P. M. Gilmer, K. Qazaqzeh,  \emph{The parity of the Maslov index
and the even cobordism catergory}, \textsl{Fund. Math.}, \textbf{184}
(2005), 95-102.

\bibitem[M]{M}
H. Murakami, \emph{Quantum SO(3)-invariants dominate the
SU(2)-invariant of Casson and Walker}, \textsl{Math. Proc.
Camb. Phil. Soc.} \textbf{117} (1995), no.2 237-249.

\bibitem[MR]{GR97}
G. Masbaum, J. D. Roberts, \emph{A simple proof of integrality of
quantum invariants at prime roots of unity}, \textsl{Math. Proc.
Camb. Phil. Soc.} \textbf{121} (1997), 443-454.

\bibitem[Q]{Q}
K. Qazaqzeh, \emph{Integral Bases for certain TQFT-Modules of the torus},
\textsl{Math. Proc. Camb. Phil. Soc.} \textbf{143} (2007), 669-684.
\end{thebibliography}

\end{document}